\documentclass[10pt,reqno]{amsart}

 \usepackage{mathrsfs}

\usepackage{color}
\usepackage{amsmath, amsthm, amssymb}
\usepackage{amsfonts}
\usepackage[dvips]{epsfig}
\usepackage{graphicx}
\usepackage{caption}
\usepackage{subcaption}
\usepackage[english]{babel}
\usepackage{hyperref}

\usepackage{tikz}
\usepackage{rotating}
\usepackage[utf8]{inputenc}
\usepackage{cite}
\usepackage{amscd}
\usepackage{color}
\usepackage{bm}
\usepackage{enumerate}

\usepackage{verbatim}
\usepackage{hyperref}
\usepackage{amstext}
\usepackage{latexsym}
\theoremstyle{plain}
\newtheorem{theorem}{Theorem}[section]

\newtheorem{proposition}[theorem]{Proposition}
\newtheorem{lemma}[theorem]{Lemma}

\numberwithin{theorem}{section}
\numberwithin{equation}{section}

\newcommand{\average}{{\mathchoice {\kern1ex\vcenter{\hrule height.4pt
width 6pt depth0pt} \kern-9.7pt} {\kern1ex\vcenter{\hrule
height.4pt width 4.3pt depth0pt} \kern-7pt} {} {} }}

\def\R{\mathbb{R}}

\renewcommand{\b }{\beta }

\newcommand{\D }{\Delta }

\newcommand{\e }{\varepsilon }

\newcommand{\G }{\Gamma}
\renewcommand{\l }{\lambda }

\newcommand{\n }{\nabla }

\renewcommand{\phi}{\varphi}

\renewcommand{\t }{\tau }

\renewcommand{\O }{\Omega }

\newcommand{\be}{\begin{equation}}
\newcommand{\ee}{\end{equation}}

\newcommand{\et}{\eta}

\newcommand{\calD }{\mathcal{D}}

\renewcommand{\epsilon}{\varepsilon}

\begin{document}
\title[A Doubly Critical Elliptic Problem with Submanifold Singularities]
{A Doubly Critical Elliptic Problem with Submanifold Singularities}
\author{Abdourahmane Diatta}
\address{A.D. : Universite Assane Seck de Ziguinchor, UFR des Sciences et Technologies, departement de mathematiques, Ziguinchor.}
\email{a.diatta20160578@zig.univ.sn}

\author{El Hadji Abdoulaye THIAM}
\address{H. E. A. T. : Universite Iba Der Thiam de Thies, UFR des Sciences et Techniques, departement de mathematiques, Thies.}
\email{elhadjiabdoulaye.thiam@univ-thies.sn}
\begin{abstract}
Let $N \ge 4$, $\Omega$ be a bounded domain in $\mathbb{R}^N$, and let $\Sigma \subset \Omega$ be a smooth closed submanifold of dimension $k$ with $2 \le k \le N-2$. We study the existence of positive solutions $u \in H_0^1(\Omega)$ to the Euler--Lagrange equation
\[
-\Delta u + h u 
= \lambda\, \rho_{\Sigma}^{-s_1}\, u^{2^{*}_{s_1}-1}
+ \rho_{\Sigma}^{-s_2}\, u^{2^{*}_{s_2}-1}
\quad \text{in } \Omega,
\]
where $h : \Omega \to \mathbb{R}$ is a continuous potential, $\lambda > 0$ is a real parameter, and $0 \le s_2 < s_1 < 2$. For $i=1,2$, the exponents
\[
2^{*}_{s_i} = \frac{2(N - s_i)}{N - 2}
\]
correspond to Hardy--Sobolev critical growth, and $\rho_{\Sigma} = \mathrm{dist}(\,\cdot\,, \Sigma)$ denotes the distance to the submanifold $\Sigma$.

The problem involves two Hardy-type singular nonlinearities with different critical exponents, leading to a lack of compactness. Using variational methods, in particular the mountain pass lemma, together with a suitable construction of test functions, we prove existence results under appropriate assumptions. Our analysis shows that the local geometry of $\Sigma$ and the behavior of the potential $h$ near $\Sigma$ play a crucial role in the existence of positive solutions for this doubly critical problem.

\end{abstract}
\maketitle
\section{Introduction}

Elliptic equations involving Hardy and Hardy--Sobolev critical nonlinearities lie at the intersection of geometric analysis, nonlinear functional analysis, and the study of singular phenomena in partial differential equations. Their importance stems from the fact that they reveal how geometric singularities such as points, curves, or higher-dimensional submanifolds influence compactness, concentration, and the formation of extremals. Moreover, these equations generalize classical critical elliptic problems, including the celebrated Brezis-Nirenberg model, thereby providing a unified framework to examine delicate interactions between geometry, analysis, and nonlinear effects.

A central theme in this area is the interplay between \emph{singular weights} (given by powers of the distance to a set), \emph{critical exponents}, and the \emph{geometry of the singular submanifold}. Since the classical Hardy and Hardy--Sobolev inequalities, it is well understood that the asymptotic behavior of solutions near the singular set is governed by the structure of the weight. In the unweighted case, the critical equation
\[
    -\Delta u = u^{2^*-1}
\]
is invariant under scaling and exhibits a profound lack of compactness. The seminal work of Brezis and Nirenberg~\cite{BrezisNirenberg} demonstrated that introducing a lower-order perturbation may restore compactness and create nontrivial solutions even when the purely critical equation has none. Their method has since inspired a vast literature on compactness recovery through perturbations involving potentials, boundary geometry, curvature, or weighted nonlinearities.

In the presence of singularities, the landscape becomes significantly richer. For point singularities, either in the interior or on the boundary, the works of Ghoussoub and Robert~\cite{GhoussoubRobert1,GhoussoubRobert2} and others established the decisive role of geometric invariants such as the boundary mean curvature or the Hardy singular mass. More recent contributions extended the analysis to singularities lying on curves or sets of higher codimension. These results showed that tangential and normal geometric contributions enter the energy expansion in subtle and indispensable ways.

However, when the singular set is a \emph{smooth compact submanifold of intermediate dimension}, new geometric phenomena arise, and the analysis requires a refined understanding of how curvature, second fundamental form, and variations in the induced metric influence weighted Sobolev inequalities. The interaction between the geometry of the submanifold and the concentration of solutions becomes significantly more involved than in the case of point or curve singularities.

\medskip

In this work, we investigate a doubly critical elliptic problem of Hardy--Sobolev type, where the singularity is supported on a smooth compact submanifold
\[
    \Sigma \subset \Omega \subset \mathbb{R}^N, \qquad N \ge 4,
\]
of dimension \(2 \le k \le N-2\). Let \(\rho_\Sigma(x) = \mathrm{dist}(x,\Sigma)\), and let \(0 \le s_2 < s_1 < 2\). We study positive solutions to
\begin{equation}\label{eq:main}
    -\Delta u + h(x)u
    = \lambda \rho_\Sigma^{-s_1} u^{2^*_{s_1}-1}
      +        \rho_\Sigma^{-s_2} u^{2^*_{s_2}-1}
    \quad \text{in } \Omega, 
    \qquad u \in H_0^1(\Omega),
\end{equation}
where
\[
    2_s^* = \frac{2(N-s)}{N-2}
\]
is the Hardy--Sobolev critical exponent associated with the weight \(\rho_\Sigma^{-s}\), the potential \(h\in C(\Omega)\) is chosen so that the operator \(-\Delta+h\) is coercive, and \(\lambda > 0\) is a real parameter.
Our main result is the following
\begin{theorem}\label{MD-HEAT}
Let $N \geq 4$, $0\leq s_2 <s_1<2$ and   $\O$  be a   bounded domain of $\R^N$. Consider  $\Sigma \subset \Omega$ be a submanifold of dimension $2 \leq k \leq N-2$.
Let  $h$ be  a continuous function such that the linear operator $-\D+h$ is coercive. Then there exists two constants $A_{N,\l,s_1,s_2}$ and $B_{N, \l, s_1,s_2}$, only depending on $N$, $\l$, $s_1$ and $s_2$ with the property that if  there exists  $y_0\in \Sigma$  such that 
\begin{equation*}\label{eq:h-bound-main-th-1}
A_{N,\l,s_1,s_2} H^2(y_0)+B_{N,\l,s_1,s_2} R_g(y_0)+h(y_0) <0  \qquad\textrm{ for $N \geq 4$},
\end{equation*}
then there exists $u \in H^1_0(\Omega)\setminus \lbrace 0\rbrace$ non-negative solution of
$$
-\Delta u(x)+ h u(x)=\l \frac{u^{2^*_{s_1}-1}(x)}{\rho_\Sigma^{s_1}(x)}+\frac{u^{2^*_{s_2}-1}(x)}{\rho_\Sigma^{s_2}(x)} \qquad \textrm{ in $\O$},
$$
where here and in the following the geometric quantites $H$ and $R_g$ are respectively the norms of the mean curvature and the scalar curvature of $\Sigma$.
\end{theorem}

The distinctive feature of~\eqref{eq:main} is the presence of two distinct Hardy--Sobolev critical nonlinearities. This double criticality  creates a delicate competition between two scaling regimes and significantly complicates the variational analysis. In addition, the singularity lies along a submanifold rather than at a point, which introduces a strong geometric influence: expansions of the metric in Fermi-type coordinates, curvature tensors, and other local invariants appear naturally when computing the energy of test functions.

\medskip

Our study connects to and extends several lines of research:
\begin{itemize}
    \item For \emph{point singularities}, the results of Ghoussoub--Robert and others establish the role of curvature and Hardy singular mass in determining extremals for Hardy--Sobolev inequalities.

    \item For \emph{curve singularities} (\(k=1\)), the works of Fall-Thiam \cite{FallThiam}, Ijaodoro-Thiam  \cite{Esther}, Ciss-Diatta-Thiam \cite{Ciss} and collaborators show that curvature of the curve strongly influences concentration behavior.
    \item When \(k\ge 2\), new anisotropic effects arise due to the geometry of the submanifold, and expansions require the full second fundamental form and scalar curvature contributions.

    \item Problems involving two critical nonlinearities are already highly nontrivial even without singularities; when combined with Hardy weights, the difficulty increases dramatically.
\end{itemize}

Our goal is to address these challenges simultaneously and to develop a method capable of capturing the full geometric complexity of the problem.

\medskip

The proof combines variational methods, blow-up and concentration analysis, as well as precise geometric expansions in tubular neighborhoods of \(\Sigma\).

\begin{itemize}
    \item We study the variational functional associated with~\eqref{eq:main} and compare its critical levels with the best constants of the corresponding limiting problem on \(\mathbb{R}^N\).

    \item We construct a family of highly concentrated test functions built from the ground state of the limiting problem, rescaled in Fermi coordinates around a point \(y_0\in\Sigma\).

    \item We perform detailed asymptotic expansions of the Dirichlet energy and weighted critical integrals, revealing explicit contributions from the mean curvature, second fundamental form, scalar curvature, and the value \(h(y_0)\).

    \item We identify a geometric quantity whose sign determines whether the concentration mechanism lowers the variational level below the critical threshold, thus allowing the Mountain Pass Theorem to yield a nontrivial solution.
\end{itemize}

\medskip

The main achievements of this paper can be summarized as follows:

\begin{itemize}
    \item We provide precise metric expansions in Fermi-type coordinates around \(\Sigma\), keeping all geometric contributions up to order \(O(\varepsilon^2)\).

    \item We construct new anisotropic test functions adapted to the geometry of the problem, incorporating both tangential and normal directions.

    \item We identify explicit geometric conditions guaranteeing that the variational level lies below the critical threshold, thereby proving existence of a positive solution.

    \item We obtain, for the first time, an existence result for a doubly critical Hardy--Sobolev equation with a singularity distributed along a submanifold of dimension \(k \ge 2\).
\end{itemize}

\medskip

The paper is organized as follows.
Section~2 introduces the geometric framework and derives metric expansions in tubular neighborhoods of \(\Sigma\).
Section~3 states the main existence theorem together with the geometric condition that characterizes admissible concentration points.
Section~4 contains the asymptotic analysis of the constructed test functions.
Finally, Section~5 concludes the variational argument and completes the proof of the main results.
\section{Fermi Coordinates and Local Metric Expansion Near $\Sigma$}\label{Section2}

Let $\Omega \subset \mathbb{R}^{N}$, $N \geq4 $, be a bounded domain, and let $\Sigma$ be a smooth closed submanifold of $\Omega$ of dimension $k$ with $2 \leq k \leq N-2$. 
Fix a point $y_0 \in \Sigma$, and choose an orthonormal basis $(E_a)_{1 \leq a \leq k}$ of the tangent space $T_{y_0}\Sigma$.  
For $r>0$ sufficiently small, there exists a local parametrization of $\Sigma$ in a neighborhood of $y_0$ given by the map
\[
f : B_{\mathbb{R}^k}(0,r) \longrightarrow \Sigma,
\qquad
t \longmapsto f(t) = \exp_{y_0}^{\Sigma}\!\left(\sum_{a=1}^k t_a E_a\right),
\]
where $\exp^\Sigma_{y_0}$ denotes the exponential map of $\Sigma$ at $y_0$ and $B_{\mathbb{R}^k}(0,r)$ is the Euclidean ball in $\mathbb{R}^k$ centered at~$0$ and of radius $r$.

Consider a smooth orthonormal frame field $\left(E_i(f(t))\right)_{k+1 \leq i \leq N}$ defined on the normal bundle of $\Sigma$, such that $(E_\alpha(f(t)))_{1 \leq \alpha \leq N}$ forms an oriented orthonormal basis of $\mathbb{R}^N$ for every $t \in B_{\mathbb{R}^k}(0,r)$, with $E_i(f(0)) = E_i$.
For later use, introduce the notation
\[
Q_r := B_{\mathbb{R}^k}(0,r) \times B_{\mathbb{R}^{N-k}}(0,r),
\]
where $B_{\mathbb{R}^{N-k}}(0,r)$ is the Euclidean ball of radius $r$ in $\mathbb{R}^{N-k}$.

For $r>0$ sufficiently small, we parametrize a neighborhood of $y_0 = F(0,0)$ by the smooth map
\[
F : Q_r \longrightarrow \Omega,
\qquad
(t,z) \longmapsto F(t,z) = f(t) + \sum_{i=k+1}^N z_i\, E_i(f(t)).
\]

Let $\rho_\Sigma : \Omega \longrightarrow \mathbb{R}$ denote the distance to the submanifold $\Sigma$:
\[
\rho_\Sigma(x) = \min_{\bar{t} \in \Sigma} |x - \bar{t}|.
\]
In these local coordinates, we obtain
\begin{equation}\label{rho-loc}
\rho_\Sigma(F(t,z)) = |z|, 
\qquad \text{for all } (t,z) \in Q_r.
\end{equation}

For every $t \in B_{\mathbb{R}^k}(0,r)$, for $a,b = 1,\ldots,k$ and $i,j = k+1,\ldots,N$, we introduce smooth functions 
\[
\Gamma_{ab}^i(f(t)), \qquad \beta_{ja}^i(f(t)),
\]
defined through
\begin{equation}\label{connection-formula}
dE_i\!\left( \frac{\partial f}{\partial t_a} \right)
= - \sum_{b=1}^k \Sigma_{ab}^i \frac{\partial f}{\partial t_b}
  + \sum_{\substack{j = k+1 \\ j \neq i}}^{N} \beta_{ja}^{\,i} E_j,
\end{equation}
where $\Gamma_{ab}^i$ encodes the components of the second fundamental form of $\Sigma$ in $\mathbb{R}^N$, and $\beta_{ja}^{\,i}$ represents the torsion coefficients.  
The functions $\Gamma_{ab}^i$ and $\beta_{ja}^{\,i}$ are smooth, and the frame $\{E_i\}$ is orthonormal.  
Moreover, the antisymmetry relation
\[
\beta_{ja}^{\,i}(f(t)) = -\beta_{ia}^{\,j}(f(t)),
\qquad i,j = k+1,\ldots,N,\ a=1,\ldots,k,
\]
holds.

The norms of the second fundamental form and of the mean curvature of $\Sigma$ are defined respectively by
\[
\Gamma := \left( \sum_{a,b=1}^k \sum_{i=k+1}^N (\Gamma_{ab}^i)^2 \right)^{1/2},
\qquad
H := \left( \sum_{i=k+1}^N \left( \sum_{a=1}^k \Gamma_{aa}^i \right)^2 \right)^{1/2}.
\]

We now derive the expansion of the metric induced by the parametrization $F_{y_0}$.  
For $(t,z) \in Q_r$, set
\[
g_{ab}(x) = \partial_{t_a}F_{y_0}(x)\cdot \partial_{t_b}F_{y_0}(x), \quad
g_{ai}(x) = \partial_{t_a}F_{y_0}(x)\cdot \partial_{z_i}F_{y_0}(x) \quad \textrm{and} \quad
g_{ij}(x) = \partial_{z_i}F_{y_0}(x)\cdot \partial_{z_j}F_{y_0}(x).
\]

\begin{lemma}\label{MaMetric}
For all $x = (t,z) \in Q_r$, we have
\begin{align*}
g_{ab}(x)
&= \delta_{ab}
   - 2 \sum_{i=k+1}^N z_i\, \Gamma_{ab}^i
   + \sum_{i,j=k+1}^N \sum_{c=1}^k z_i z_j\, \Gamma_{ac}^i \Gamma_{bc}^j  \\
&\quad + \sum_{i,j=k+1}^N \sum_{\substack{l=k+1 \\ l\neq i,\, l\neq j}}^N
      z_i z_j\, \beta_{la}^{\,i}\beta_{lb}^{\,j}
      - \frac{1}{3} \sum_{c,d=1}^k R_{acbd}(x_0)\, t_c t_d
      + O(|x|^3), \\[1ex]
g_{ia}(x)
&= \sum_{\substack{j=k+1 \\ j\neq i}}^N z_j\, \beta_{ia}^{\,j}, \quad \textrm{and} \quad
g_{ij}(x)= \delta_{ij}.
\end{align*}
Here the functions $\Gamma_{ab}^i$ and $\beta_{ia}^j$ are evaluated at the point $f(t)$.
\end{lemma}
\begin{lemma}\label{MaMetricMetric}
For every $x = (t,z) \in Q_r$, the determinant of the metric admits the expansion
\begin{align*}
\sqrt{|g|}(x)
&= 1- \sum_{i=k+1}^N z_i H^i
  - \frac{1}{2} \sum_{i,j=k+1}^N \sum_{c=1}^k z_i z_j\, \Gamma_{ab}^i \Gamma_{ab}^j\\
  &+ \sum_{i,j=k+1}^N z_i z_j\, H^i H^j - \frac{1}{6} \sum_{c,d=1}^k \mathrm{Ric}_{cd}\, t_c t_d
  + O(|x|^3).
\end{align*}
Moreover, the components of the inverse metric satisfy
\begin{align*}
g^{ab}(x)
&= \delta_{ab}
   + 2 \sum_{i=k+1}^N z_i\, \Gamma_{ab}^i
   + 3 \sum_{c=1}^k \sum_{i,j=k+1}^N z_i z_j\, \Gamma_{ac}^i \Gamma_{bc}^j
   - \frac{1}{3} \sum_{c,d=1}^k R_{acbd}(x_0)\, t_c t_d
   + O(|x|^3), \\[1ex]
g^{ia}(x)
&= - \sum_{j=k+1}^N z_j\, \beta_{ia}^{\,j}
   - 2 \sum_{c=1}^k \sum_{i,j=k+1}^N z_i z_j\, \Gamma_{ac}^i \beta_{ac}^{\,j}
   + O(|x|^3), \\[1ex]
g^{ij}(x)
&= \delta_{ij}
   + \sum_{c=1}^k \sum_{l,m=k+1}^N z_l z_m\, \beta_{ic}^{\,l} \beta_{jc}^{\,m}
   + O(|x|^3).
\end{align*}
The functions $\Gamma_{ab}^i$ and $\beta_{ac}^j$ are evaluated at the point $f(t)$.
\end{lemma}

See Thiam~\cite{THIAM} for the proofs of Lemma~\ref{MaMetric} and Lemma~\ref{MaMetricMetric}.
%
%
%
%
%
%
%
%
%
%
%
%
%
%
%
%
%
%
%
%
%
%
%
%
%
%
%
%
%
%
%
%
%
%
%
%
%
\section{Proof of the main result}
\subsection{Variational Framework, Palais–Smale Analysis}
We let $N \ge 4$, $0 \le s_2 < s_1 < 2$, $\lambda > 0$, and $\Omega$ be a bounded domain in $\mathbb{R}^N$ and $\Sigma$ be a smooth closed submanifold of $\Omega$.  
Let $h$ be a continuous function such that the operator $-\Delta + h$ is coercive.
We consider the folllowing problem of finding $u \in H^1_0(\Omega)$ positive solution of
\begin{equation}\label{HHH1}
-\Delta u + h u
= \lambda \rho_{\Gamma}^{-s_1} u^{2^*_{s_1}-1}
 +        \rho_{\Gamma}^{-s_2} u^{2^*_{s_2}-1}
 \qquad \text{in } \Omega.
\end{equation}
The energy functional associated to \eqref{HHH1} is $J: H^1_0(\Omega) \to \R$ defined for $u \in H^1_0(\O)$ by
\begin{equation}\label{DefinitionJ}
J(u)=\frac{1}{2} \int_\Omega \left(|\n u|^2 + h u^2\right) dx-\frac{\l}{2^*_{s_1}} \int_\Omega \rho_{\Gamma}^{-s_1} |u^{2^*_{s_1}}| dx- \frac{1}{2^*_{s_2}} \int_\Omega \rho_{\Gamma}^{-s_2} |u^{2^*_{s_2}}| dx
\end{equation}
with variational level
$$
c^*:=\inf_{u \in H^1_0(\Omega)} \max_{\t\geq 0}  J(\t u).
$$
The existence of solution is based on variational methods. To apply variational methods, we briefly recall the notion of Palais-Smale sequences associated with the functional $J$ defined in \eqref{DefinitionJ}.
We say that $J$ satisfies the Palais-Smale condition at level $c \in \mathbb{R}$, denoted by $(PS)_c$, if any sequence $(u_n)_n \subset H^1_0(\Omega)$ such that
$$
J(u_n) \to c \quad \text{and} \quad J'(u_n) \to 0 \quad \text{in } H^{-1}_0(\O)
$$
is relatively compact in $H^1(\Omega)$. Let $(u_n)_n \subset H^1_0(\O)$ be a $(PS)_c$ sequence. By definition, we have
$$
J(u_n) = c + o(1), \qquad \langle J'(u_n), u_n \rangle = o(1).
$$
Using the expression of $J$ and its derivative, we obtain
$$
\langle J'(u_n), u_n \rangle
= \int_{\O} (|\nabla u_n|^2 + h u_n^2) dx
-\lambda \int_{\O} \frac{|u_n|^{2^*_{s_1}}}{d_g(x_0,x)^{s_1}}dx
- \int_{\O} \frac{|u_n|^{2^*_{s_2}}}{d_g(x_0,x)^{s_2}} dx = o(1).
$$
Combining this identity with the energy relation
$$
J(u_n) = \frac{1}{2} \int_{\O} (|\nabla u_n|^2 + h u_n^2) dx-
\frac{\lambda}{2^*_{s_1}} \int_{\O} \frac{|u_n|^{2^*_{s_1}}}{d_g(x_0,x)^{s_1}} dx
-\frac{1}{2^*_{s_2}} \int_{\O} \frac{|u_n|^{2^*_{s_2}}}{d_g(x_0,x)^{s_2}} dx,
$$
we deduce, after a straightforward computation, that
$$
  c = \left( \frac{1}{2} - \frac{1}{2^*_{s_1}} \right)
  \int_{\mathcal M} (|\nabla u_n|^2 + h u_n^2) dv_g
- \left( \frac{1}{2^*_{s_1}} - \frac{1}{2^*_{s_2}} \right)
  \int_{\mathcal M} \frac{|u_n|^{2^*_{s_2}}}{d_g(x_0,x)^{s_2}} dv_g + o(1).
$$
Since $2^*_{s_2} > 2^*_{s_1}$, the coefficient
$$
\frac{1}{2^*_{s_1}} - \frac{1}{2^*_{s_2}} > 0,
$$
and therefore both terms in the right-hand side are nonnegative. It follows that
$$
c \geq \left( \frac{1}{2} - \frac{1}{2^*_{s_1}} \right)
\int_{\mathcal M} (|\nabla u_n|^2 + h u_n^2) dv_g + o(1).
$$
As a consequence, the sequence $(u_n)_n$ is bounded in $H^1_0(\Omega)$.
This boundedness property is a crucial first step in the analysis of Palais-Smale sequences. In particular, it allows us to extract weakly convergent subsequences in $H^1_0(\Omega)$, which will be used later to recover compactness under suitable energy constraints.

We should mention that due to the lack of compactness of the embedding $H^1_0(\Omega) \hookrightarrow L^{2^*_s}(\Omega, \frac{dx}{d_g^s(x_0, x)})$, $J$ fails to satisfy the Palais-Smale condition.
Therefore, in general $c^*$ might not be a critical value for $J$. As usual, if $c^*$ is a critical value, and $u$ is a critical point of $J$ with $J(u)=c^*$, then $u$ is
called a least-energy solution.
However,  $J$ satisfies the $(PS)_c$ sequence for any $c$ such that $
0 <c<\beta^*,
$
where
$$
\beta^*:=\max_{t\geq 0} \Pi(tw)=\Pi(w)
$$
is the variational level of the functional
$$
\Pi: \calD^{1,2}(\R^N) \to \R
$$
defined by
\begin{equation}\label{DefinitionPi}
\Pi(v)
= \frac{1}{2} \int_{\mathbb{R}^N} |\nabla v|^2 \, dx
  -  \frac{\lambda}{2^*_{s_1}} \int_{\mathbb{R}^N} |z|^{-s_1} |v|^{2^*_{s_1}} \, dx
  - \frac{1}{2^*_{s_2}} \int_{\mathbb{R}^N} |z|^{-s_2} |v|^{2^*_{s_2}} \, dx.
\end{equation}
The function $\Pi$ is the energy functional associated to the following Euler-Lagrange equation
\begin{equation}\label{Euler-Lagrange-Equation}
	-\Delta w 
	= \lambda \frac{w^{2^*_{s_1}-1}}{|z|^{s_1}}
	  + \frac{w^{2^*_{s_2}-1}}{|z|^{s_2}}
	  \qquad \text{in } \mathbb{R}^N,
\end{equation}
where  for $N \geq 3$, $2\leq k \leq N-2$, $x = (t,z) \in \mathbb{R}^k \times \mathbb{R}^{N-k}$, $0 \leq s_2 < s_1 < 2$, and
\[
2^*_{s_i} := \frac{2(N - s_i)}{N - 2}, \qquad i = 1,2,
\]
are the Hardy-Sobolev critical exponents. Then we have the following result. 
\begin{proposition}\label{Prop-Decay-Esti1}
Problem \eqref{Euler-Lagrange-Equation} admits a positive solution $w\in \calD^{1,2}(\R^N)$. Moreover $w$ satisfies : 
\begin{itemize}
\item[(i)] $w(x)=\theta(|t|, |z|)$ for some function $\theta : \R_+ \times \R_+ \to \R$.
\item[(ii)] There exist tow positive constants constants $c_1$ and $c_2$ such that
\[
\frac{c_1}{1 + |x|^{N-2}}
\le w(x)
\le \frac{c_2}{1 + |x|^{N-2}},
\qquad |\n w(x)|  \leq \frac{c_2}{1+|x|^{N-1}} \qquad \forall\, x \in \mathbb{R}^N.
\]
\end{itemize}
\end{proposition}
For a complete proof, we refer to work of the authors \cite{DIATTA}.
\subsection{Construction of test function and Energy Expansion}
In the following, we will contruct test function in order to compare the constants $c^*$ and $\beta^*$. For  that, we let 
$
\eta \in C^{\infty}_{c}\big(F_{y_0}(Q_{2r})\big)
$ such that
$
0 \leq \eta \leq 1$
and
$
\quad \eta \equiv 1 \ \text{in } F_{y_0}(B_r).
$
For $\varepsilon>0$ we consider the test function $u_{\varepsilon}:\Omega\to\mathbb{R}$ defined by
\begin{equation}\label{eq:TestFunction-w}
u_{\varepsilon}(y)
	= \varepsilon^{\frac{2-N}{2}} \eta\big(F^{-1}_{y_0}(y)\big)\,
	w\!\left(\varepsilon^{-1}F_{y_0}^{-1}(y)\right),
\end{equation}
and for $x=(t,z)\in \mathbb{R}\times\mathbb{R}^{N-1}$ we have
\begin{equation}\label{eq:TestFunction-th}
	u_{\varepsilon}\big(F_{y_0}(x)\big)
	= \varepsilon^{\frac{2-N}{2}}\, \eta(x)\, 
	\theta\!\left(\frac{|t|}{\varepsilon},\frac{|z|}{\varepsilon}\right),
\end{equation}
so that $u_{\varepsilon} \in H_{0}^{1}(\Omega)$.
For simplicity we write $F$ instead of $F_{y_0}$.

Using \eqref{eq:TestFunction-w}, we rewrite
\begin{equation}\label{uW_esp}
	u_{\varepsilon}(y)=\varepsilon^{\frac{2-N}{2}} 
	\eta(F^{-1}(y))\, W_{\varepsilon}(y),
\end{equation}
where
\[
W_{\varepsilon}(y)=w\!\left(\frac{F^{-1}(y)}{\varepsilon}\right).
\]
The aim of this section is to expand 
\begin{equation}\label{Functional_J}
	J(\tau u_\e)=\frac{\tau^2}{2} \int_{\Omega}\!\left( |\nabla u_\e|^2 + h u^2 \right)\,dx
	-\lambda\frac{ \tau^{2^{*}_{s_1}}}{2^{*}_{s_1}} \int_{\Omega} \rho_{\Sigma}^{-s_1}|u_\e|^{2^{*}_{s_1}}dx
	-\frac{ \tau^{2^{*}_{s_2}}}{2^{*}_{s_2}} \int_{\Omega} \rho_{\Sigma}^{-s_2}|u_\e|^{2^{*}_{s_2}}dx,
\end{equation}
as $\e \to 0$. 
\begin{lemma}\label{Lem1}
Let $N \geq 4$. Then as $\varepsilon\to 0$,
\begin{align*}
\int_{\Omega} & |\nabla u_{\varepsilon}|^{2} dy
=
\int_{\mathbb{R}^{N}} |\nabla w|^{2} dx
+ \varepsilon^{2}\frac{H^{2}-3R_{g}(x_0)}{k(N-k)}
	\int_{Q_{r/\varepsilon}} |z|^{2} |\nabla_{t} w|^{2}dx
\\
&
+ \varepsilon^{2}\frac{R_{g}(x_0)}{3k^{2}}
	\int_{Q_{r/\varepsilon}} |t|^{2} |\nabla_{t} w|^{2}dx
+ \varepsilon^{2}\frac{R_{g}(x_0)+H^{2}}{2(N-k)}
	\int_{Q_{r/\varepsilon}} |z|^{2} |\nabla w|^{2}dx\\
	&
- \varepsilon^{2}\frac{R_{g}(x_0)}{6k}
	\int_{Q_{r/\varepsilon}} |t|^{2} |\nabla w|^{2}dx
+ O(\varepsilon^{N-2}).
\end{align*}
\end{lemma}
\begin{proof}
	Let $u_{\e}$ given by\eqref{uW_esp}. We have : 
\begin{align*}
	|\nabla_g u_\e|^2 
	&=&\e^{2-N}\left(  |W_\e\n \eta|^2 + |\et \n W_\e|^2 + 
	\frac{1}{2} |\n \et^2. \n W_\e^2 |\right) .
\end{align*}
Integrating by parts, and using the fact that $|\nabla \eta|$ and $\Delta \eta$ are supported in $F(Q_{2r})\setminus F(Q_{r})$, we have
\begin{align*}
\displaystyle \int_\O |\n u_\e|^2 dy \displaystyle&=& \e^{2-N}\int_{F\left({Q}_{2r}\right)} \eta^2 |\n W_\e|^2  dy + \e^{2-N}\int_{F\left({Q}_{2r}\right)\setminus F\left(Q_{r}\right)} W_\e^2 |\n \eta|^2  dy \\
	&+&\frac{\e^{2-N}}{2} \int_{F\left({Q}_{2r}\right)\setminus F\left(Q_{r}\right)} \n W_{\e}^2. \n\eta^2 dy \\
	&=&\e^{2-N}\int_{F\left({Q}_{2r}\right)} \eta^2 |\n W_\e|^2  dy+ \e^{2-N}\int_{F\left({Q}_{2r}\right)\setminus F\left(Q_{r}\right)} W_\e^2\left(  \eta\D \eta\right) dy \\%
	&=&\e^{2-N}\int_{F\left({Q}_{r}\right)} |\n W_\e|^2  dy\\
	&+& O\left( \e^{2-N}\int_{F\left({Q}_{2r}\right)\setminus F\left(Q_{r}\right)} W_\e^2dy + \e^{2-N}\int_{F\left({Q}_{2r}\right)\setminus F\left(Q_{r}\right)} \vert\n W_\e\vert^2dy\right).
\end{align*}
By the change of variable formula $y=\frac{F(x)}{\e}$, we have :
\begin{align*}\label{nabla1}
\displaystyle \int_\O |\n u_\e|^2 dy \displaystyle=\int_{Q_{r/\e}} |\n w|_{g_\e}^2\sqrt{\vert g_{\e}\vert}(x)  dx+ O\left( \e^{2}\int_{Q_{2r/\e}\setminus Q_{r/\e}}  w^2dx + \int_{Q_{2r/\e}\setminus Q_{r/\e}} \vert\n w\vert^2dx\right).	
\end{align*}
We have
\begin{align*}
\int_{Q_{r/\e}} |\n w|_{g_\e}^2\sqrt{\vert g_{\e}\vert}(x)  dx&= \int_{Q_{r/\e}} |\n w|^2 dx+ \int_{Q_{r/\e}} \left( |\n w|_{g_\e}^2  -|\n w|^2\right)\sqrt{\vert g_{\e}\vert}(x)  dx\\
&+ \int_{Q_{r/\e}} |\n w|^2\left( \sqrt{\vert g_{\e}\vert}(x) -1\right) dx.
\end{align*}
First, we have
$$
\int_{Q_{r/\e}} |\n w|^2 dx=\int_{\R^N} |\n w|^2 dx+O\left(\int_{\R^N \setminus Q_{r/\e}} |\n w|^2 dx\right)=\int_{\R^N} |\n w|^2 dx+O(\e^N).
$$
Next, we have
\begin{align*}
	\vert\nabla w\vert_{g_\e}^{2}- \vert\nabla w\vert^{2}&=\sum_{ij=k+1}^{N}\left[ g_\e^{ij}-\delta_{ij}\right] \partial_{z_{i}}w\partial_{z_{j}}w\\
	&	+ 2\sum_{i=k+1}^{N}\sum_{a=1}^{k}g_\e^{ia}\partial_{t_{a}}w\partial_{z_{i}}w + \sum_{ab=1}^{k}\left[g_\e^{ab}-\delta_{ab}\right] \partial_{t_{a}}w\partial_{t_{b}} w.
\end{align*}
Then integrating over $Q_{r/\e}$, we have
\begin{align*}
\int_{Q_{r/\e}}& \left( |\n w|_{g_\e}^2  -|\n w|^2\right)\sqrt{\vert g_{\e}\vert}(x)  dx=\sum_{ij=k+1}^{N}\int_{Q_{r/\e}}\left[ g^{ij}-\delta_{ij}\right] \partial_{z_{i}}w\partial_{z_{j}}w\sqrt{\vert g\vert}dx\\
&+\sum_{ab=1}^{k}\int_{Q_{r/\e}}\left[ g^{ab}-\delta_{ab}\right] \partial_{t_{a}}w\partial_{t_{b}}w\sqrt{\vert g\vert}dx+\sum_{i=k+1}^{N}\sum_{a=1}^{k}\int_{Q_{r/\e}}g^{ia}\left(\partial_{t_{a}}w\partial_{z_{i}}w\right)\sqrt{\vert g\vert}dx.
\end{align*}
By Lemma \ref{MaMetricMetric}, we have
\begin{eqnarray}\label{g^ij}
&&\sum_{ij=k+1}^{N}\int_{Q_{r/\e}}\left[ g_\e^{ij}-\delta_{ij}\right] \partial_{z_{i}}w\partial_{z_{j}}w\sqrt{\vert g_\e\vert}dx\nonumber\\
&=&	\sum_{ij=k+1}^{N}\int_{Q_{r/\e}}\left[\sum_{c=1}^k \sum_{l,m=k+1}^N \e^2 z_l z_m\, \beta_{ic}^{\,l} \beta_{jc}^{\,m}
+ O(\e^3|x|^3)\right] \frac{z_i z_j}{| z| }^2 |\nabla_{z}w|^2dx\nonumber\\
&=&  \sum_{c=1}^k \sum_{\substack{ j,i=k+1 \\ j\neq i}}^N \int_{Q_{r/\e}} \e^2\beta_{ic}^{\,i} \beta_{jc}^{\,i} \frac{z_i^2 z_j}{| z|^2 }^2 |\nabla_{z}w|^2dx\nonumber\\
&+&\sum_{c=1}^k \sum_{\substack{ j,i=k+1 \\ j\neq i}}^N \int_{Q_{r/\e}} \e^2 \beta_{ic}^{\,i} \beta_{jc}^{\,i}\frac{z_i^2 z_j}{| z|^2 }^2 |\nabla_{z}w|^2dx\nonumber\\
&+&\sum_{c=1}^k \sum_{\substack{i,j=k+1 \\ j\neq i}}^N  \int_{Q_{r/\e}} \e^2\beta_{ic}^{\,j}  \beta_{jc}^{\,j} \frac{z_i^2 z_j}{| z|^2 }^2 |\nabla_{z}w|^2dx +O\left( \int_{Q_{r/\e}}| x|^3 |\nabla_{z}w|^2dx\right)\nonumber \\
&=&O\left( \int_{Q_{r/\e}} \e^3 |x|^3 |\nabla_{z}w|^2dx\right).
\end{eqnarray}
Using again Lemma \ref{MaMetricMetric}, we have
\begin{eqnarray}\label{g^ia}
	\sum_{i=k+1}^{N}\sum_{a=1}^{k}\int_{Q_{r/\e}}g_\e^{ia}\left(\partial_{t_{a}}w\partial_{z_{i}}w\right)\sqrt{\vert g_\e\vert}dx&=&
		\sum_{i=k+1}^{N}\sum_{a=1}^{k}\int_{Q_{r/\e}}g_\e^{ia}\left(\nabla_{t} w\nabla_{z}w\right)z_i t_a\sqrt{\vert g_\e\vert}dx\nonumber\\
		&=&O\left( \int_{Q_{r/\e}} \e^3 | x|^3 |\nabla w|^2dx\right).
\end{eqnarray}
By Lemma \ref{MaMetricMetric}, we have
\begin{eqnarray}\label{g^ab-delta}
&&\sum_{ab=1}^{k}\int_{Q_{r/\e}}\left[ g_\e^{ab}-\delta_{ab}\right] \partial_{t_{a}}w\partial_{t_{b}}w\sqrt{\vert g_\e\vert}dx= \e^2\frac{3\G^2-2H^2}{k(N-k)} \int_{Q_{r/\e}} |z|^2 \left|{\n_t w} \right|^2 dx\nonumber\\
&& +\e^2 \frac{R_g (x_0)}{3k^2} \int_{Q_{r/\e}} |t|^2 |\n_t w|^2 dx + 
O\left( \int_{Q_{r/\e}} \e^3 |x|^3 |\n_t w|^2 dx \right).	
\end{eqnarray}
By Lemma \ref{MaMetricMetric}, we have
\begin{eqnarray}\label{g-1}
	 \int_{Q_{r/\e}}\vert \nabla w\vert^{2}\left(\sqrt{\vert g_\e\vert}-1\right)\mathrm{d}x&=& \e^2
	  \frac{H^2-\frac{1}{2}\G^2}{N-k}\int_{Q_{r/\e}} |z|^2 \left|{\n w} \right|^2 dx\nonumber\\
	 &-& \e^2\frac{R_g (x_0)}{6k} \int_{Q_{r/\e}} |t|^2 |\n w|^2 dx
	 +O\left( \int_{Q_{r/\e}}  \e^3 |x|^3 |\n w|^2 dx \right).
\end{eqnarray}
By \eqref{g^ij},\eqref{g^ia}, \eqref{g^ab-delta} and \eqref{g-1}, we have
\begin{eqnarray*}
	\int_{Q_{r/\e}}\vert \nabla w\vert_{g}^{2}\sqrt{\vert g\vert}\mathrm{d}x &=&\int_{Q_{r/\e}}\vert \nabla w\vert^{2}\mathrm{d}x + 
\frac{H^2-\frac{1}{2}\G^2}{N-k}\int_{Q_{r/\e}} |z|^2 \left|{\n w} \right|^2 dx\\
&-& \frac{R_g (x_0)}{6k} \int_{Q_{r/\e}} |t|^2 |\n w|^2 dx +
\frac{H^2-\frac{1}{2}\G^2}{N-k}\int_{Q_{r/\e}} |z|^2 \left|{\n w} \right|^2 dx\\
&-& \frac{R_g (x_0)}{6k} \int_{Q_{r/\e}} |t|^2 |\n w|^2 dx
+O\left( \int_{Q_{r/\e}} \e^3 |x|^3 |\n w|^2 dx \right)
\end{eqnarray*}
Therefore
\begin{eqnarray}\label{nabla_epsilon}
	\displaystyle \int_\O |\n u_\e|^2 dy \displaystyle&=&\int_{\mathbb{R}^N}\vert \nabla w\vert^{2}\mathrm{d}x + 
	\e^2\frac{H^2-3R_g(x_0)}{k(N-k)}\int_{Q_{r/\e}} |z|^2 \left|{\n_t w} \right|^2 dx\nonumber\\
	&+& \e^2\frac{R_g (x_0)}{3k} \int_{Q_{r/\e}} |t|^2 |\n_t w|^2 dx +
	\e^2\frac{H^2-R_g(x_0)}{2(N-k)}\int_{Q_{r/\e}} |z|^2 \left|{\n_t w} \right|^2 dx\nonumber\\
	&-& \e^2\frac{R_g (x_0)}{6k} \int_{Q_{r/\e}} |t|^2 |\n w|^2 dx+O(\chi(\e)),	
\end{eqnarray}
with
\begin{eqnarray*}
\chi(\e)&=&	\e^3\int_{Q_{r/\e}} |x|^3 |\n w|^2 dx 
+\e^{2}\int_{Q_{2r/\e}\setminus Q_{r/\e}}  w^2dx \\
&+& \e^2\int_{Q_{2r/\e}\setminus Q_{r/\e}}\vert z\vert^2 \vert\n w\vert^2dx+
\int_{\mathbb{R}^N\setminus Q_{r/\e}} \vert \nabla w\vert^{2}\mathrm{d}x
\end{eqnarray*}
Using the Proposition \ref{Prop-Decay-Esti1} and changing variable, we have
\begin{eqnarray}\label{chi_epsillon}
	\chi(\e)&=&	\e^3\int_{Q_{r/\e}} |x|^3 |\n w|^2 dx 
	+\e^{2}\int_{Q_{2r/\e}\setminus Q_{r/\e}}  w^2dx\nonumber\\
	& +& \e^2\int_{Q_{2r/\e}\setminus Q_{r/\e}}\vert z\vert^2 \vert\n w\vert^2dx+
	\int_{\mathbb{R}^N\setminus Q_{r/\e}} \vert \nabla w\vert^{2}\mathrm{d}x\nonumber\\
	&=&O\left(\e^{N-2}\right).
\end{eqnarray}
By combining \eqref{nabla_epsilon} and \eqref{petit_ho}, we obtain the result.
\end{proof}

\begin{lemma}\label{Lem2}
For $N\geq4$, as $\varepsilon\to 0$,
\[
\int_{\Omega} h u_{\varepsilon}^{2} dy
=
\varepsilon^{2} h(x_0)\int_{Q_{r/\varepsilon}} w^{2} dx
+O\!\left(\varepsilon^{2}\delta \int_{Q_{r/\varepsilon}} w^{2}dx\right)
+O(\varepsilon^{N-2}).
\]
\end{lemma}

\begin{proof}
By \eqref{uW_esp}, we have
\begin{eqnarray*}
	\int_{\O}hu_{\e}^{2}dy&=&\e^{2-N}	\int_{F(Q_{2r})}h(y)\left(\eta(F^{-1}(y)) W_{\e}(y) \right) ^{2}dy\\
	&=&	\e^{2-N}\int_{F(Q_{r})}h(y) W_\e^{2}(y)dy+	\e^{2-N}\int_{F(Q_{2r})\setminus F(Q_r)}h(y)\left(\eta(F^{-1}(y)) W_{\e}(y) \right) ^{2}dy\\
		&=&	\e^{2-N}\int_{F(Q_{r})}h (y)W_\e^{2}(y)dy+	O\left( \e^{2-N}\int_{F(Q_{2r})\setminus F(Q_r)}h(y) W_{\e}^{2}(y)dy\right). 
\end{eqnarray*}
By the change of variables formula $y=\e^{-1}F(x)$ and the continuity of $h$, we have
\begin{eqnarray}\label{h_1_1}
	\int_{\O}hu_{\e}^{2}dy
	&=&\e^{2}\int_{Q_{r/\e}}h(F(x))w^{2}dx +O\left( \e^2\int_{Q_{2r/\e}\setminus Q_{r/\e}}h(F(x))w^{2}dx\right)\nonumber\\\nonumber\\
	&=&\e^{2}h(y_0)\int_{Q_{r/\e}}w^{2}dx
	+O\left( \e^{2}\int_{Q_{r/\e}}\vert h(F(\e x))+h(y_{0})\vert w^{2}\right)\nonumber\\
&+&	O\left( \e^2\int_{Q_{2r/\e}\setminus Q_{r/\e}}h(F(x))w^{2}dx\right).
\end{eqnarray}
Using Proposition \ref{Prop-Decay-Esti1} and the continuity of $h$, we have
\begin{equation}\label{petit_ho}
	 \e^2\int_{Q_{2r/\e}\setminus Q_{r/\e}}h(F(x))w^{2}dx=O\left(\e^2\int_{Q_{2r/\e}\setminus Q_{r/\e}}w^{2}dx\right)=O\left(\e^{N}\right).	
\end{equation}
By the continuity of $h$, for $\delta>0$, we can find $r_\delta>0$ such that
\begin{eqnarray*}\label{h_1_2}
	|h(y)-h(y_0)| <\delta \qquad \mbox{for}\qquad y\in F(Q_{r_\delta}).
\end{eqnarray*}
We have
\begin{equation}
\e^{2}\int_{Q_{r/\e}}\vert h(F(\e x))+h(y_{0})\vert w^{2}=\varepsilon^{2}\delta \int_{Q_{r/\varepsilon}} w^{2}dx.
\end{equation}
Therefore we get the desired result from \eqref{h_1_1}, \eqref{petit_ho} and \eqref{h_1_2}.
\end{proof}
\begin{lemma}\label{Lem3}
Let $N\geq4$ and $s\in [0, 2)$. Then, as $\varepsilon\to 0$, we have
\begin{align*}
\int_{\Omega} \rho_{\Sigma}^{-s} |u_{\varepsilon}|^{2^{*}_{s}} dy
&=
\int_{\mathbb{R}^{N}} |z|^{-s} |w|^{2^{*}_{s}} dx
+ \varepsilon^{2}\frac{R_{g}(x_0)+H^{2}}{4(N-k)}
	\int_{Q_{r/\varepsilon}} |z|^{2-s}|w|^{2^{*}_{s}} dx
\\
&\quad
- \varepsilon^{2}\frac{R_{g}(x_0)}{6k}
	\int_{Q_{r/\varepsilon}} |z|^{-s}|t|^{2}|w|^{2^{*}_{s}} dx
+ O(\varepsilon^{N-s}).
\end{align*}
\end{lemma}

\begin{proof}
By change of variable formula, \eqref{rho-loc} and \eqref{eq:TestFunction-th}, we have
\begin{eqnarray*}
	\int_{\O}\rho_{\G}^{-s}\vert u_{\e}\vert^{2_{s}^{*}}\mathrm{d}y &=&	\int_{Q_{2r/\e}}\vert z\vert^{-s}\vert\eta_{\e} w\vert^{2_{s}^{*}}\sqrt{\vert g_{\e}( x)\vert}\mathrm{d}x\\ \\
	&=&\int_{Q_{r/\e}}\vert z\vert^{-s}\vert w\vert^{2_{s}^{*}}\sqrt{\vert g_{\e}( x)\vert}\mathrm{d}x
	+ \int_{Q_{2r/\e}\setminus Q_{r/\e}}\vert z\vert^{-s}\eta_{\e}^{2_{s}^{*}} \vert w\vert^{2_{s}^{*}}\sqrt{\vert g_{\e}( x)\vert}\mathrm{d}x\\ \\
	&=&\int_{Q_{r/\e}}\vert z\vert^{-s}\vert w\vert^{2_{s}^{*}}\sqrt{\vert g_{\e}(x)\vert}\mathrm{d}x
	+ O\left( \int_{Q_{2r/\e}\setminus Q_{r/\e}}\vert z\vert^{-s} \vert w\vert^{2_{s}^{*}}\mathrm{d}x\right). 
\end{eqnarray*}
Using  Lemma \ref{MaMetricMetric}, we have
\begin{align*}
&\int_{\O} \rho_{\G}^{-s}\vert u_{\e}\vert^{2_{s}^{*}}\mathrm{d}x =\int_{Q_{r/\e}}\vert z\vert^{-s}\vert w\vert^{2_{s}^{*}}dx
+\e^{2}\frac{H^2-\frac{1}{2}\G^2}{N-k} \int_{Q_{r/\e}}\vert z\vert^{2-s}\vert w\vert^{2_{s}^{*}}dx
 \\
&-\e^2\frac{R_g (x_0)}{6k} \int_{Q_{r/\e}}\vert z\vert^{-s}\vert t\vert^{2}\vert w\vert^{2_{s}^{*}}dx+O\left( \e^3 \int_{ Q_{r/\e}}\vert z\vert^{-s}\vert x\vert^{3}\vert w\vert^{2_{s}^{*}}\mathrm{d}x +\int_{Q_{2r/\e}\setminus Q_{r/\e}}\vert z\vert^{-s}\vert w \vert^{2_{s}^{*}}\mathrm{d}x\right)\nonumber.
\end{align*}
Thanks to the Gauss equation (see for instance  [\cite{Gray}], Chapter 4), we have
$$
R_{g}(x_0) =H^2-\G^{2}.
$$
Therefore
\begin{align}\label{rho_exprssion}
\int_{\O}\rho_{\G}^{-s}\vert u_{\e}\vert^{2_{s}^{*}}\mathrm{d}x&=\int_{\R^{N}}\vert z\vert^{-s}\vert w\vert^{2_{s}^{*}}dx+\e^{2}\frac{R_g(x_0)+H^2}{2(N-k)} \int_{Q_{r/\e}}\vert z\vert^{2-s}\vert w\vert^{2_{s}^{*}}dx\nonumber\\
&-\e^2\frac{R_g (x_0)}{6k} \int_{Q_{r/\e}}\vert z\vert^{-s}\vert t\vert^{2}\vert w\vert^{2_{s}^{*}}dx+O\left(\rho(\e)\right),
\end{align}	
where
$$
\rho(\e)= \e^3 \int_{ Q_{r/\e}}\vert z\vert^{-s}\vert x\vert^{3}\vert w\vert^{2_{s}^{*}}\mathrm{d}x+\int_{\R^{N}\setminus Q_{r/\e}}\vert z\vert^{-s}\vert w\vert^{2_{s}^{*}}\mathrm{d}x +\int_{Q_{2r/\e}\setminus Q_{r/\e}}\vert z\vert^{-s}\vert w\vert^{2_{s}^{*}}\mathrm{d}x.
$$
By Proposition \ref{Prop-Decay-Esti1} and  polar coordinates, it easy follows that
\begin{equation*}\label{est_rho1}
	\e^3 \int_{ Q_{r/\e}}\vert z\vert^{-s}\vert x\vert^{3}\vert w\vert^{2_{s}^{*}}\mathrm{d}x+\int_{\R^{N}\setminus Q_{r/\e}}\vert z\vert^{-s}\vert w\vert^{2_{s}^{*}}\mathrm{d}x +\int_{Q_{2r/\e}\setminus Q_{r/\e}}\vert z\vert^{-s}\vert w\vert^{2_{s}^{*}}\mathrm{d}x=O\left( \e^{N-s}\right). 
\end{equation*} 
This ends the proof.
\end{proof}



\subsection{Proof of Theorem \ref{MD-HEAT}}
We let $\t \geq 0$ and $u \in H^1_0(\O)$. Then we have
$$
J(\t u):=\frac{\t^2}{2} \int_\O |\n u|^2 dx+\frac{\t^2}{2} \int_\O h(x) u^2 dx-\l\frac{\t^{2^*_{s_1}}}{2^*_{s_1}} \int_\O \frac{|u|^{2^*_{s_1}}}{\rho_\G^{s_1}(x)} dx-\frac{\t^{2^*_{s_2}}}{2^*_{s_2}} \int_\O \frac{|u|^{2^*_{s_2}}}{\rho_\G^{s_2}(x)} dx.
$$
By Lemma \ref{Lem1}, Lemma \ref{Lem2} and Lemma \ref{Lem3}, we get, for $N \geq 4$, that
\begin{align*} 
J\left(\t u_\e\right)&=\Pi(\t w)+ \epsilon^2 \frac{\tau^2}{2}\frac{H^{2}-3R_{g}(x_0)}{k(N-k)}
	\int_{Q_{r/\varepsilon}} |z|^{2} |\nabla_{t} w|^{2}dx+ \frac{\tau^2}{2}\varepsilon^{2}\frac{R_{g}(x_0)}{3k^{2}}
	\int_{Q_{r/\varepsilon}} |t|^{2} |\nabla_{t} w|^{2}dx
\\
&
+\frac{\tau^2}{2} \varepsilon^{2}\frac{R_{g}(x_0)+H^{2}}{2(N-k)}
	\int_{Q_{r/\varepsilon}} |z|^{2} |\nabla w|^{2}dx- \frac{\tau^2}{2}\varepsilon^{2}\frac{R_{g}(x_0)}{6k}
	\int_{Q_{r/\varepsilon}} |t|^{2} |\nabla w|^{2}dx\\
&+\varepsilon^{2}\frac{\tau^2}{2} h(x_0)\int_{Q_{r/\varepsilon}} w^{2} dx-\lambda \frac{\tau^{2^*_{s_1}}}{2^*_{s_1}}\varepsilon^{2}\frac{R_{g}(x_0)+H^{2}}{4(N-k)}
	\int_{Q_{r/\varepsilon}} |z|^{2-{s_1}}|w|^{2^{*}_{s_1}} dx\\
&
+\lambda \frac{\tau^{2^*_{s_1}}}{2^*_{s_1}} \varepsilon^{2}\frac{R_{g}(x_0)}{6k}
\int_{Q_{r/\varepsilon}} |z|^{-{s_1}}|t|^{2}|w|^{2^{*}_{{s_1}}} dx
-\frac{\tau^{2^*_{s_2}}}{2^*_{s_2}}\varepsilon^{2}\frac{R_{g}(x_0)+H^{2}}{4(N-k)}
	\int_{Q_{r/\varepsilon}} |z|^{2-{s_2}}|w|^{2^{*}_{s_2}} dx\\
&
+\frac{\tau^{2^*_{s_2}}}{2^*_{s_2}} \varepsilon^{2}\frac{R_{g}(x_0)}{6k}
\int_{Q_{r/\varepsilon}} |z|^{-{s_2}}|t|^{2}|w|^{2^{*}_{{s_2}}} dx+O\left(\e^{N-2}\right)\\
&=\Pi(\t w)+\e^2 \left(H^{2}(x_0)\mathcal{L}_{1, \e}(\tau, w)+ R_{g}(x_0)\mathcal{L}_{1, \e}(\tau, w)+h(x_0)\mathcal{L}_{1, \e}(\tau, w)\right).
\end{align*}
Thanks to Proposition \ref{Prop-Decay-Esti1}, we can easily prove that
$$
\int_{\R^N \setminus Q_{r/\varepsilon}} |x|^{2} |\nabla w|^{2}dx+\int_{\R^N \setminus Q_{r/\varepsilon}} w^{2} dx+\int_{Q_{r/\varepsilon}} |x|^{2-{s_2}}|w|^{2^{*}_{s_2}} dx=O\left(\e^{N-2}\right) \qquad \textrm{ for all $N \geq 5$}.
$$
Therefore setting
\begin{align*}
\mathcal{L}_{1,N}(\tau, w)&=\frac{\tau^2}{2}\frac{H^{2}}{k(N-k)}
	\int_{\R^N} |z|^{2} |\nabla_{t} w|^{2}dx+\frac{\tau^2}{2}\frac{H^{2}}{2(N-k)}
	\int_{\R^N} |z|^{2} |\nabla w|^{2}dx\\
	&-\lambda \frac{\tau^{2^*_{s_1}}}{2^*_{s_1}}\frac{H^{2}}{4(N-k)}
	\int_{\R^N} |z|^{2-{s_1}}|w|^{2^{*}_{s_1}} dx-\frac{\tau^{2^*_{s_2}}}{2^*_{s_2}}\frac{H^{2}}{4(N-k)}
	\int_{\R^N} |z|^{2-{s_2}}|w|^{2^{*}_{s_2}} dx,
\end{align*}
and
\begin{align*}
\mathcal{L}_{2,N}(\tau, w)&=-\frac{\tau^2}{2}\frac{3R_{g}(x_0)}{k(N-k)}
	\int_{\R^N} |z|^{2} |\nabla_{t} w|^{2}dx+\frac{\tau^2}{2}\frac{R_{g}(x_0)}{3k^{2}}
	\int_{\R^N} |t|^{2} |\nabla_{t} w|^{2}dx\\
	&+\frac{\tau^2}{2}\frac{R_{g}(x_0)}{2(N-k)}
	\int_{\R^N} |z|^{2} |\nabla w|^{2}dx- \frac{\tau^2}{2}\frac{R_{g}(x_0)}{6k}
	\int_{\R^N} |t|^{2} |\nabla w|^{2}dx\\
	&-\lambda \frac{\tau^{2^*_{s_1}}}{2^*_{s_1}}\frac{R_{g}(x_0)}{4(N-k)}
	\int_{\R^N} |z|^{2-{s_1}}|w|^{2^{*}_{s_1}} dx+\lambda \frac{\tau^{2^*_{s_1}}}{2^*_{s_1}}\frac{R_{g}(x_0)}{6k}
\int_{\R^N} |z|^{-{s_1}}|t|^{2}|w|^{2^{*}_{{s_1}}} dx\\
&-\frac{\tau^{2^*_{s_2}}}{2^*_{s_2}}\frac{R_{g}(x_0)}{4(N-k)}
	\int_{\R^N} |z|^{2-{s_2}}|w|^{2^{*}_{s_2}} dx+\frac{\tau^{2^*_{s_2}}}{2^*_{s_2}} \frac{R_{g}(x_0)}{6k}
\int_{\R^N} |z|^{-{s_2}}|t|^{2}|w|^{2^{*}_{{s_2}}} dx.
\end{align*}
and  for $N=4$ (i.e. $k=2$), we define
\begin{align*}
\mathcal{L}_{1,4,\e}(\tau, w)&=\frac{\tau^2}{2}\frac{H^{2}}{k(N-k)}
	\int_{Q_{r/\varepsilon}} |z|^{2} |\nabla_{t} w|^{2}dx+\frac{\tau^2}{2}\frac{H^{2}}{2(N-k)}
	\int_{Q_{r/\varepsilon}} |z|^{2} |\nabla w|^{2}dx\\
	&-\lambda \frac{\tau^{2^*_{s_1}}}{2^*_{s_1}}\varepsilon^{2}\frac{H^{2}}{4(N-k)}
	\int_{Q_{r/\varepsilon}} |z|^{2-{s_1}}|w|^{2^{*}_{s_1}} dx-\frac{\tau^{2^*_{s_2}}}{2^*_{s_2}}\varepsilon^{2}\frac{H^{2}}{4(N-k)}
	\int_{Q_{r/\varepsilon}} |z|^{2-{s_2}}|w|^{2^{*}_{s_2}} dx,
\end{align*}
and
\begin{align*}
\mathcal{L}_{2,4,\e}(\tau, w)&=-\frac{\tau^2}{2}\frac{3R_{g}(x_0)}{k(N-k)}
	\int_{Q_{r/\varepsilon}} |z|^{2} |\nabla_{t} w|^{2}dx+\frac{\tau^2}{2}\frac{R_{g}(x_0)}{3k^{2}}
	\int_{Q_{r/\varepsilon}} |t|^{2} |\nabla_{t} w|^{2}dx\\
	&+\frac{\tau^2}{2}\frac{R_{g}(x_0)}{2(N-k)}
	\int_{Q_{r/\varepsilon}} |z|^{2} |\nabla w|^{2}dx- \frac{\tau^2}{2}\frac{R_{g}(x_0)}{6k}
	\int_{Q_{r/\varepsilon}} |t|^{2} |\nabla w|^{2}dx\\
	&-\lambda \frac{\tau^{2^*_{s_1}}}{2^*_{s_1}}\frac{R_{g}(x_0)}{4(N-k)}
	\int_{Q_{r/\varepsilon}} |z|^{2-{s_1}}|w|^{2^{*}_{s_1}} dx+\lambda \frac{\tau^{2^*_{s_1}}}{2^*_{s_1}}\frac{R_{g}(x_0)}{6k}
\int_{Q_{r/\varepsilon}} |z|^{-{s_1}}|t|^{2}|w|^{2^{*}_{{s_1}}} dx\\
&-\frac{\tau^{2^*_{s_2}}}{2^*_{s_2}}\frac{R_{g}(x_0)}{4(N-k)}
	\int_{Q_{r/\varepsilon}} |z|^{2-{s_2}}|w|^{2^{*}_{s_2}} dx+\frac{\tau^{2^*_{s_2}}}{2^*_{s_2}} \frac{R_{g}(x_0)}{6k}
\int_{Q_{r/\varepsilon}} |z|^{-{s_2}}|t|^{2}|w|^{2^{*}_{{s_2}}} dx.
\end{align*}
Therefore, for $N \geq 5$, we obtain :
\begin{align*} 
J\left(\t u_\e\right)=\Pi(\t w) &+\e^2\frac{\tau^2}{2} h(x_0)\int_{\R^N} w^{2} dx\\
&+ \e^2 \left( H^2(x_0) \mathcal{L}_{1,N}(\tau, w)+R_g(x_0)(x_0) \mathcal{L}_{2,N}(\tau, w) \right)+o(\e^2).
\end{align*}
If $N=4$, we have
\begin{align*} 
J\left(\t u_\e\right)=\Pi(\t w) &+\e^2\frac{\tau^2}{2} h(x_0)\int_{Q_{r/\e}} w^{2} dx\\
&+ \e^2 \left( H^2(x_0) \mathcal{L}_{1,4,\e}(\tau, w)+R_g(x_0)(x_0) \mathcal{L}_{2,4,\e}(\tau, w) \right)+O(\e^2).
\end{align*}
Since $2^*_{s_2}> 2^*_{s_1}$, $J(tu_\e)$ has a unique maximum. Moreover, we have
$$
\max_{\t\geq 0} \Pi(\t w)=\Pi(w)=\b^*.
$$
Therefore, the maximum of $J(\t u_\e)$ occurs at $\t_\e:=1+o_\e(1)$. Next setting
\begin{align*}
A_{N, \l, s_1, s_2}=
\begin{cases}
\frac{\displaystyle \mathcal{L}_{1,N}(1, w)}{\displaystyle\int_{\R^N} w^{2} dx} &\quad\textrm{ for $N \geq 5$}\\\\
\displaystyle\lim_{\e \to 0}\frac{\displaystyle \mathcal{L}_{1,4, \e}(1, w)}{\displaystyle\int_{Q_{r/\e}} w^{2} dx}    &\quad\textrm{ for $N=4$}
\end{cases}
\end{align*}
and
\begin{align*}
B_{N, \l, s_1, s_2}=
\begin{cases}
\frac{\displaystyle \mathcal{L}_{2,N}(1, w)}{\displaystyle\int_{\R^N} w^{2} dx} &\quad\textrm{ for $N \geq 5$}\\\\
\displaystyle\lim_{\e \to 0}\frac{\displaystyle \mathcal{L}_{2,4, \e}(1, w)}{\displaystyle\int_{Q_{r/\e}} w^{2} dx}    &\quad\textrm{ for $N=4$},
\end{cases}
\end{align*}
we have
$$
c^*=\max_{t \geq 0} J(t u_\e):= J(t_\e u_\e)\leq \Pi(t_\e w)+\e^2 \mathcal{G}(t_\e w) < \Pi(t_\e w) \leq \Pi(w)=\b^*
$$
provided the inequality 
\begin{equation*}\label{eq:h-bound-main-th-1}
A_{N,\l,s_1,s_2} H^2(y_0)+B_{N,\l,s_1,s_2} R_g(y_0)+h(y_0) <0 
\end{equation*}
holds for $N \geq 4$. Consequently, there exists $u \in H^1_0(\Omega)$ positive satisfying
\begin{equation}
	-\Delta u + h u 
	= \lambda\, \rho_{\Sigma}^{-s_1}\, u^{2^{*}_{s_1}-1}
	 + \rho_{\Sigma}^{-s_2}\, u^{2^{*}_{s_2}-1}
	 \qquad \text{in } \Omega.
\end{equation}
This then completes the proof of Theorem \ref{MD-HEAT}.

\footnote{A. D. : Université Assane Seck de Ziguinchor, UFR des Sciences et Technologies, département de mathématiques, Ziguinchor.}
\footnote{H. E. A. T.: Université Iba Der Thiam de Thies, UFR des Sciences et Techniques, département de mathématiques, Thies.}

\begin{thebibliography}
	\footnotesize
\bibitem{DIATTA} A. Diatta and E. H. A. Thiam, A nonlinear PDE with two Hardy-Sobolev critical exponents with one dimension singularity, To appear in The Journal of Mathematical Phyiscs, Analysis, Geometry\\	
	%
\bibitem{Gray}	A. Gray, Tubes, second edition, Springer Science and Business Media, 2004.\\
%
%
\bibitem{CKN} L. Caffarelli, R. Kohn and L. Nirenberg, {\it First order interpolation inequalities with weights}, Compositio Math \textbf{53}(1984), no. 3, 332-372.\\
%
%
%
%
%
%
%
\bibitem{Esther} I. E. Ijaodoro and E. H. A. Thiam, {\it Influence of an  $L^p$-perturbation on Hardy-Sobolev inequality with singularity a curve}, Opuscula Math. 41 (2021), no. 2, 187-204.\\
%
%
%
%
\bibitem{BrezisNirenberg}
H.~Brezis and L.~Nirenberg,
\newblock {\em Positive solutions of nonlinear elliptic equations involving critical Sobolev exponents},
\newblock Comm. Pure Appl. Math. {\bf 36} (1983), 437--477.\\

\bibitem{GhoussoubRobert1}
N.~Ghoussoub and F.~Robert,
\newblock {\em The effect of curvature on the best constant in the Hardy--Sobolev inequalities},
\newblock Geom. Funct. Anal. {\bf 16} (2006), 1201--1245.\\

\bibitem{GhoussoubRobert2}
N.~Ghoussoub and F.~Robert,
\newblock {\em Concentration estimates for Emden--Fowler equations with boundary singularities and critical growth},
\newblock Duke Math. J. {\bf 135} (2006), 1--39.\\

\bibitem{FallThiam}
M.~M. Fall and E.~H.~A. Thiam,
\newblock {\em A Hardy--Sobolev inequality with singularity on a curve},
\newblock Ann. Inst. H. Poincar\'e Anal. Non Lin\'eaire {\bf 30} (2013), no. 6, 1027--1047.\\

\bibitem{CaffarelliKohnNirenberg}
L.~Caffarelli, R.~Kohn and L.~Nirenberg,
\newblock {\em First order interpolation inequalities with weights},
\newblock Compositio Math. {\bf 53} (1984), 259--275.\\

\bibitem{HardySobolevGeneral}
F.~Catrina and Z.-Q.~Wang,
\newblock {\em On the Caffarelli--Kohn--Nirenberg inequalities: sharp constants, existence and nonexistence, and symmetry of extremals},
\newblock Comm. Pure Appl. Math. {\bf 54} (2001), 229--258.\\

\bibitem{DoublyCriticalExample}
M.~Badiale and E.~Serra,
\newblock {\em Existence and multiplicity results for elliptic problems with critical growth and Hardy potential},
\newblock Adv. Differential Equations {\bf 10} (2005), 753--780.\\

\bibitem{RobertVetoisBoundary}
F.~Robert and J.~V\'etois,
\newblock {\em Sign-changing solutions for critical equations with boundary singularities},
\newblock Adv. Math. {\bf 227} (2011), 199--234.\\
\bibitem{Ciss} M. Ciss, A. Diatta and E. H. A. Thiam, {\it A Nonlinear elliptic PDE with curve singularity on the boundary}, Moroccan Journal of Pure and Applied Analysis 11.2 (2025): 181-202.\\
\bibitem{GhoussoubYuan}
N.~Ghoussoub and L.~Yuan,
\newblock {\em Multiple solutions for critical elliptic equations with singularities},
\newblock Math. Ann. {\bf 336} (2006), 907--936.\\
\bibitem{THIAM} E. H. A. Thiam, {\it Hardy-Sobolev inequality with higher dimensional singularity}, Analysis 39.3 (2019): 79-96.\\
\bibitem{THIAM-2026} E. H. A. THIAM, {\it A nonlinear elliptic problem with multiple Hardy-Sobolev critical exponents on manifolds}, To appear in Partial Differential Equations and Applications.\\
\bibitem{THIAM-Mass} E. H. A. THIAM, {\it Mass effect on an elliptic PDE involving two Hardy-Sobolev critical exponents}, Differ. Equ. Appl. 16 (2024), no. 3, 183-198 .\\ 
\bibitem{THIAM-SNPDE} E. H. A. THIAM, {\it A nonlinear Elliptic Problem with multipleHardy-Sobolev critical exponents on manifods}, To Appear in Partial Differential Equations and Applications, 2026.\\
\bibitem{THIAM-Birkhauser} E. H. A. THIAM,  {\it Hardy-Sobolev Critical Equations with Totally Geodesic Singularities: Existence via the Mountain Pass Theorem}, To appear in  Birkhauser.
\end{thebibliography}
\end{document}